\documentclass[11pt]{article}

\usepackage{amsmath, amssymb, amsthm}
\usepackage[utf8]{inputenc}
\usepackage[T1]{fontenc}
\usepackage{geometry}
\usepackage{hyperref}
\usepackage{mathtools}

\newtheorem{theorem}{Theorem}[section]
\newtheorem{lemma}[theorem]{Lemma}
\newtheorem{proposition}[theorem]{Proposition}
\newtheorem{corollary}[theorem]{Corollary}
\theoremstyle{definition}
\newtheorem{definition}[theorem]{Definition}
\theoremstyle{remark}
\newtheorem{remark}[theorem]{Remark}

\newcommand{\C}{\mathbb{C}}
\newcommand{\R}{\mathbb{R}}
\newcommand{\D}{\mathbb{D}}
\newcommand{\HH}{\mathbb{H}}
\newcommand{\Acal}{\mathcal{A}}
\newcommand{\Bcal}{\mathcal{B}}
\newcommand{\Fcal}{\mathcal{F}}
\newcommand{\Mcal}{\mathcal{M}}
\newcommand{\BV}{\mathcal{BV}}
\newcommand{\Gcal}{\mathcal{G}}
\newcommand{\Scal}{\mathcal{S}}

\title{The Absorption Theorem for the Beltrami--Vekua Normal Form}
\author{Daniel Alay\'on-Solarz\thanks{\texttt{danieldaniel@gmail.com}}}
\date{June 2026}

\begin{document}

\maketitle

\begin{abstract}
The Beltrami–Vekua normal form assigns to every smooth first-order real planar elliptic system a complex equation $w_{\bar z}-\mu w_z+\mathcal{A}w+\mathcal{B}\bar w=\mathcal{F}$ by an explicit pipeline. A companion paper showed that the density $\Theta=|\mathcal{B}|^2/(1-|\mu|^2)\,dx\,dy$ and its total mass are invariants under multiplicative gauges $w\mapsto\phi w$ and orientation-preserving diffeomorphisms. The real system carries a larger symmetry: its unknowns may be recombined by any pointwise invertible real-linear substitution $w=\varphi v'+\psi\bar v'$, the complex gauges being the case $\psi\equiv0$. We prove the absorption theorem: re-normalizing through the pipeline after any such substitution returns to the gauge orbit of the original equation, with a universal explicit gauge $\tilde\varphi=-i\lambda/(\varphi-\psi)$, where $\lambda$ is the spectral root of the structure polynomial. Hence every gauge–diffeomorphism invariant of the Beltrami–Vekua data — the density $\Theta$, the pseudo-analytic mass, the zero locus of $\mathcal{B}$ — is an invariant of the underlying real elliptic system under its full natural symmetry groupoid. The proof is structural: pairing the unknown with a $\lambda$-eigencovector of the symbol's canonical endomorphism gives an eigenpairing equation that is natural under substitutions and turns eigencovector rescalings into the gauge action; the pipeline normalization and the ambient form $w=u+iv$ are sections of one eigencovector line whose ratio is the gauge. The result extends to orientation-reversing substitutions, so the invariant theory survives recombination by $\mathrm{GL}(2,\mathbb{R})$-valued matrix fields.
\end{abstract}

\section{Introduction}\label{sec:intro}

The companion paper \cite{mass} attaches to every smooth first-order real planar elliptic system a \emph{Beltrami--Vekua} (BV) equation
\begin{equation}\label{eq:BV}
w_{\bar z} - \mu\, w_z + \Acal\, w + \Bcal\, \bar w = \Fcal, \qquad |\mu| < 1,
\end{equation}
through a seven-step pipeline that is algebraic apart from one differentiation of the principal-part coefficients, and proves that the 2-form
\begin{equation}\label{eq:Theta}
\Theta = \frac{|\Bcal|^2}{1 - |\mu|^2}\, dx\, dy
\end{equation}
is invariant under multiplicative gauges $w \mapsto \phi w$ and covariant under orientation-preserving diffeomorphisms, so that its total mass $\Mcal = \int_\Omega \Theta$, the \emph{pseudo-analytic mass}, is an invariant of the gauge--diffeomorphism class of the BV data, vanishing exactly on the analytic class $\Bcal \equiv 0$.

An invariant of \emph{what}, exactly? The BV data are produced from a real system in two unknowns $(u, v)$, and the real system carries a symmetry that the complex formalism does not display on its face: the unknowns may be recombined by an arbitrary smooth pointwise invertible real-linear substitution. In complex notation, writing $w = u + iv$, these are the maps
\begin{equation}\label{eq:gen-sub}
w \;=\; \varphi\, v' + \psi\, \bar v', \qquad |\varphi| > |\psi|,
\end{equation}
with $\varphi, \psi \in C^1(\Omega; \C)$; the inequality is pointwise invertibility with positive determinant, so \eqref{eq:gen-sub} is a $\mathrm{GL}^+(2,\R)$-valued substitution, and the complex gauges are exactly the $\C$-linear case $\psi \equiv 0$. Likewise, the two scalar equations may be recombined by an arbitrary invertible matrix of functions acting on the left. If the pseudo-analytic mass is to be an invariant of the \emph{system} rather than of its complex presentation, it must survive both operations.

This paper proves that it does, in the strongest available form. For a substitution $S = (\varphi, \psi)$ and BV data $D$, let $S^\sharp D$ denote the real system obtained by substituting \eqref{eq:gen-sub} into the equation of $D$ and separating real and imaginary parts, and let $D_S := \pi(S^\sharp D)$ be its pipeline normal form.

\begin{theorem}[Absorption Theorem]\label{thm:main}
For every $D \in \BV(\Omega)$ and every substitution $S = (\varphi, \psi)$ with $\varphi, \psi \in C^1(\Omega; \C)$ and $|\varphi| > |\psi|$ pointwise,
\begin{equation}\label{eq:universal-gauge}
D_S \;=\; \tilde\varphi \cdot D, \qquad \tilde\varphi \;=\; \frac{-i\lambda}{\varphi - \psi} \;\in\; C^1(\Omega; \C^*),
\end{equation}
where $\lambda = i(1+\mu)/(1-\mu) \in \HH^+$ is the spectral root of $D$ and $\tilde\varphi \cdot D$ denotes the gauge action \eqref{eq:gauge-action}. In particular $\mu_S = \mu$, $|\Bcal_S| = |\Bcal|$, and $\Theta_{D_S} = \Theta_D$ pointwise.
\end{theorem}

The substitution is thus not merely tolerated by the invariants: it is annihilated by the normal form itself, absorbed into a single multiplicative gauge whose closed form is universal --- it depends on the substitution and on the Beltrami coefficient of $D$, through $\lambda$, and on nothing else. Row recombinations are absorbed identically by the skeletonization step, and diffeomorphisms act covariantly as already proved in \cite{mass}; hence every gauge--diffeomorphism invariant of BV data is an invariant of the underlying smooth  real elliptic system under its full natural symmetry groupoid (Corollary~\ref{cor:invariants}). The theorem moreover extends, with the same gauge, to orientation-reversing substitutions $|\psi| > |\varphi|$ (Corollary~\ref{cor:reversal}): the entire invariant theory survives recombination by arbitrary $\mathrm{GL}(2,\R)$-valued matrix fields with nowhere-vanishing determinant.

The proof is short, and its shape is worth stating in words. The principal symbol of an elliptic system canonically defines an endomorphism field $N = \sigma(e_1)^{-1}\sigma(e_2)$ of the unknown bundle, whose characteristic polynomial is the structure polynomial of the pipeline; ellipticity makes its eigenvalue $\lambda \in \HH^+$ simple. Pairing the unknown vector with any $C^1$ left $\lambda$-eigencovector field $\omega$ of $N$ produces a complex unknown $w = \omega \cdot (u,v)$ which satisfies a BV equation with explicit data --- the \emph{eigenpairing equation} $E(\omega)$ (Lemma~\ref{lem:eigen-eq}). Three short lemmas then carry the entire argument: rescaling the eigencovector transforms $E(\omega)$ by exactly the gauge action (Lemma~\ref{lem:gauge-cov}); recombining the unknowns transports the eigencovector and leaves $E$ strictly invariant, $E(\omega S) = E(\omega)$, with the derivative contamination created by the substitution cancelling in one visible line (Lemma~\ref{lem:naturality}); and the pipeline computes $E(\omega_0)$ in its own normalization $\omega_0$ of the eigencovector, as a direct consequence of the residual identity of \cite{mass} (Proposition~\ref{prop:pipeline-id}). One last identification closes the loop: a BV equation is the eigenpairing equation of its own realification, in the ambient normalization $e = (1, i)$, $w = u + iv$ --- indeed the canonical endomorphism of the realification is nothing but multiplication by $\lambda$ (Lemma~\ref{lem:realification}). Absorption is then formal: the substitution transports the ambient eigencovector to $eS$, the pipeline selects its own eigencovector $\omega_0'$, the two are sections of the same eigencovector line because $\lambda$ is simple, and their ratio --- read off the second components --- is $i(\varphi - \psi)/\lambda$, whose reciprocal is the gauge \eqref{eq:universal-gauge}.

From this vantage the theorem says that the pipeline computes a frame-independent spectral object in a particular frame, and the gauge group of \cite{mass} appears a priori, as the normalization ambiguity of an eigencovector, rather than a posteriori as a transformation law. We return to this reading, and to the further structures it suggests, in Section~\ref{sec:concluding}.

The paper is organized as follows. Section~\ref{sec:setup} recalls the normal form, the residual identity --- the single result imported from \cite{mass} --- and the symmetry groupoid, and fixes the statement to be proved. Section~\ref{sec:eigenpairing} constructs the eigenpairing equation and proves the determinacy lemma that replaces jet realization. Section~\ref{sec:naturality} proves gauge covariance and substitution naturality. Section~\ref{sec:identifications} identifies both the pipeline output and the original BV equation as eigenpairing equations. Section~\ref{sec:proof} assembles the proof of Theorem~\ref{thm:main} and derives the corollaries. Section~\ref{sec:concluding} closes with structural remarks. All lemmas, and the closed form \eqref{eq:universal-gauge}, were additionally verified by exact rational computer algebra on generic variable-coefficient instances and generic variable substitutions of both orientation classes.

\section{The normal form, the residual identity, and the symmetry groupoid}\label{sec:setup}

We recall the objects of \cite{mass}, to which we refer for the full construction.

\subsection{Systems, pipeline, and BV data}

A \emph{smooth first-order planar elliptic system} on a domain $\Omega \subset \R^2$ is a pair of real equations in real unknowns $(u, v)$,
\begin{equation}\label{eq:real-system}
\begin{cases}
-v_y + a_{11}\, u_x + a_{12}\, u_y + a_{13}\, u + a_{14}\, v = f_1, \\[2pt]
\phantom{-}v_x + a_{21}\, u_x + a_{22}\, u_y + a_{23}\, u + a_{24}\, v = f_2,
\end{cases}
\end{equation}
with $a_{i1}, a_{i2} \in C^1(\Omega)$, lower-order data and forcings in $C^0(\Omega)$, and ellipticity $a_{11} > 0$, $a_{11}a_{22} - \tfrac14(a_{12}+a_{21})^2 > 0$. The skeleton form \eqref{eq:real-system} is reachable from any first-order $2\times 2$ real elliptic system by a left multiplication with an invertible matrix of functions; we treat that normalization as part of the pipeline, and we call a system \emph{admissible} if it is so reachable with $a_{11} > 0$.

The pipeline of \cite[\S 3]{mass} produces from \eqref{eq:real-system} the structure data $\alpha = a_{22}/a_{11}$, $\beta = -(a_{12}+a_{21})/a_{11}$, the spectral root $\lambda = \bigl(-\beta + i\sqrt{4\alpha - \beta^2}\bigr)/2 \in \HH^+$ of the structure polynomial $p(s) = s^2 + \beta s + \alpha$, the Beltrami coefficient $\mu = (\lambda - i)/(\lambda + i) \in \D$, and --- after the canonical substitution $U = a_{22}u$, $V = v - a_{12}u$, the obstruction correction, the fiber-algebra bundling, and the spectral and Cayley conversions --- the BV equation \eqref{eq:BV} with data $D = (\mu, \Acal, \Bcal, \Fcal)$, $\mu \in C^1$, $\Acal, \Bcal, \Fcal \in C^0$. We write $\pi$ for this normalization map. The single property of $\pi$ that
the proof below uses is the \emph{residual identity}: writing $R_1, R_2$ for the
residuals of the two equations \eqref{eq:real-system} and
$\mathrm{Res}_{\pi(L)}(w) = w_{\bar z} - \mu w_z + \Acal w + \Bcal \bar w - \Fcal$
for the residual of the output, one has, \emph{identically in the first jet of}
$(u,v)$,
\begin{equation}\label{eq:residual-identity}
\mathrm{Res}_{\pi(L)}(w) \;=\; \tfrac12(1-\mu)\bigl[-\lambda^2 R_1 + \lambda R_2\bigr], \qquad w = U + \lambda V .
\end{equation}
Both sides are affine functions of the first jet of $(u,v)$ at each point, and the
identity holds between these affine functions, not merely on solutions. It is a
direct consequence of the explicit pipeline of \cite[\S 3]{mass}; since
\cite{mass} does not display it in this form, we record the short derivation.

\begin{lemma}[Residual identity]\label{lem:residual}
For every admissible system $L$ in skeleton form \eqref{eq:real-system}, the
pipeline output $\pi(L)$ satisfies \eqref{eq:residual-identity}.
\end{lemma}

\begin{proof}
We track the residual through the steps of \cite[\S 3]{mass}; every identity
invoked is displayed there, and the only computation added is the spectral
evaluation at the end. Throughout, equalities are between affine functions of the
first jet of $(u,v)$.

\emph{Canonical form \cite[Step~3]{mass}.} The substitution $U = a_{22}u$,
$V = v - a_{12}u$ followed by the row operations of \cite[Step~3]{mass} ---
multiply the first skeleton equation by $\alpha$, and add $\beta$ times the first
to the second --- carries the skeleton residuals to the canonical residuals
$\rho_1 := \alpha R_1$ and $\rho_2 := R_2 + \beta R_1$, since the row operations
act on residuals by the same linear combinations. With the canonical coefficients
$a,b,c,d$ and forcings $f,g$ of \cite[Step~3]{mass},
\[
U_x - \alpha V_y + aU + bV - f = \rho_1, \qquad
V_x + U_y - \beta V_y + cU + dV - g = \rho_2 .
\]

\emph{Bundling \cite[Steps~4--5]{mass}.} Set $W := U + \mathfrak{i}\,V$ in the
fiber algebra $\mathcal{A}_z = \R[X]/(X^2 + \beta X + \alpha)$,
$\mathfrak{i} := [X]$. The product-rule identity of \cite[Step~5]{mass},
\[
\tfrac12(\partial_x + \mathfrak{i}\,\partial_y)\,W
= \tfrac12\bigl[(U_x - \alpha V_y) + \mathfrak{i}(U_y + V_x - \beta V_y) + G\,V\bigr],
\qquad G = \mathfrak{i}_x + \mathfrak{i}\,\mathfrak{i}_y = G_1 + G_2\,\mathfrak{i},
\]
is exact. Substituting the two canonical-form relations for the bracketed
principal combinations and folding $G_1 V, G_2 V$ into the modified coefficients
$\tilde b = b - G_1$, $\tilde d = d - G_2$ of \cite[Step~4]{mass} gives
\[
\tfrac12(\partial_x + \mathfrak{i}\,\partial_y)\,W
= \tfrac12\bigl(\rho_1 + \mathfrak{i}\,\rho_2\bigr)
 - \tfrac12\bigl[(aU + \tilde b V) + \mathfrak{i}(cU + \tilde d V)\bigr]
 + \tfrac12(f + \mathfrak{i}\,g).
\]
By construction the fiber-Vekua data $A, B, F$ of \cite[Step~5]{mass} satisfy
$A W + B\widehat W = \tfrac12[(aU + \tilde b V) + \mathfrak{i}(cU + \tilde d V)]$ and
$F = \tfrac12(f + \mathfrak{i}\,g)$, where $\widehat W = U + \hat{\mathfrak{i}}\,V$,
$\hat{\mathfrak{i}} = -\beta - \mathfrak{i}$. Hence the last two groups equal
$-(A W + B\widehat W) + F$, and
\[
\tfrac12(\partial_x + \mathfrak{i}\,\partial_y)\,W + A W + B\widehat W - F
\;=\; \tfrac12\bigl(\rho_1 + \mathfrak{i}\,\rho_2\bigr).
\]

\emph{Spectral and Cayley conversion \cite[Steps~6--7]{mass}.} Apply the algebra
homomorphism $\mathcal{A}_z \to \C$, $\mathfrak{i} \mapsto \lambda$, which commutes
with the derivatives via $\mathfrak{i}_x \mapsto \lambda_x$, $\mathfrak{i}_y \mapsto \lambda_y$;
then $W \mapsto w = U + \lambda V$, $\widehat W \mapsto \bar w$, the left side
becomes the spectral-Vekua residual
$\mathrm{Res}_{\mathrm{spec}}(w) := \tfrac12(\partial_x + \lambda\partial_y)w
+ A_\lambda w + B_\lambda \bar w - F_\lambda$ of \cite[Step~6]{mass}, and the right
side becomes
\[
\tfrac12\bigl[\alpha R_1 + (R_2 + \beta R_1)\lambda\bigr]
= \tfrac12\bigl[(\alpha + \beta\lambda)R_1 + \lambda R_2\bigr]
= \tfrac12\bigl[-\lambda^2 R_1 + \lambda R_2\bigr],
\]
using $\alpha + \beta\lambda = -\lambda^2$ from $\lambda^2 + \beta\lambda + \alpha = 0$.
Finally the Cayley conversion of \cite[Step~7]{mass} uses
$\tfrac12(\partial_x + \lambda\partial_y) = \tfrac{1}{1-\mu}(\partial_{\bar z} - \mu\partial_z)$
and $(\Acal,\Bcal,\Fcal) = (1-\mu)(A_\lambda, B_\lambda, F_\lambda)$, so that
$\mathrm{Res}_{\pi(L)}(w) = (1-\mu)\,\mathrm{Res}_{\mathrm{spec}}(w)$. Combining,
\[
\mathrm{Res}_{\pi(L)}(w)
= (1-\mu)\cdot\tfrac12\bigl[-\lambda^2 R_1 + \lambda R_2\bigr]
= \tfrac12(1-\mu)\bigl[-\lambda^2 R_1 + \lambda R_2\bigr],
\]
which is \eqref{eq:residual-identity}.
\end{proof}

In particular $\pi$ is an equivalence: since $\tfrac12(1-\mu) \neq 0$ and
$-\lambda^2 R_1 + \lambda R_2 = 0$ with $R_1, R_2$ real forces $R_1 = R_2 = 0$
(as $\lambda \notin \R$), the output residual vanishes exactly when both skeleton
residuals do, so $(u,v)$ solves \eqref{eq:real-system} if and only if $w$ solves
\eqref{eq:BV}.

Let $\BV(\Omega)$ denote the BV data on $\Omega$, with the gauge action of $\Gcal(\Omega) = C^1(\Omega; \C^*)$,
\begin{equation}\label{eq:gauge-action}
\phi \cdot D : \qquad \mu' = \mu, \qquad \Acal' = \Acal - \frac{\phi_{\bar z}}{\phi} + \mu\,\frac{\phi_z}{\phi}, \qquad \Bcal' = \Bcal\,\frac{\phi}{\bar\phi}, \qquad \Fcal' = \phi\,\Fcal,
\end{equation}
which is the transformation law of \eqref{eq:BV} under $w = \phi\, w'$, and the diffeomorphism pullback of \cite[\S 5]{mass}. The density \eqref{eq:Theta}, the mass $\Mcal$, the zero locus and dichotomy of $\Bcal$, are invariant under \eqref{eq:gauge-action} and (covariantly) under orientation-preserving diffeomorphisms \cite{mass}.

\subsection{The symmetry groupoid and the statement}

\begin{definition}[Symmetry groupoid]\label{def:groupoid}
The \emph{natural symmetry groupoid} $\Scal$ of smooth first-order planar elliptic systems is generated by:
\begin{itemize}
\item[(R)] \emph{Row recombination}: left multiplication of the pair of equations by $M \in C^0(\Omega; \mathrm{GL}(2,\R))$;
\item[(U)] \emph{Unknown recombination}: the substitution $(u, v) = S(x,y) \cdot (u', v')$ with $S \in C^1(\Omega; \mathrm{GL}^+(2,\R))$, equivalently \eqref{eq:gen-sub} in complex notation;
\item[(D)] \emph{Coordinate change}: pullback by an orientation-preserving $C^1$ diffeomorphism between domains.
\end{itemize}
\end{definition}

The regularity split mirrors that of the pipeline: substitutions enter the principal part and are differentiated once, hence $C^1$; row recombinations never are, hence $C^0$. In complex notation a substitution is the pair $(\varphi, \psi)$ of \eqref{eq:gen-sub}, related to the matrix by $S = [\varphi] + [\psi]\begin{psmallmatrix} 1 & 0 \\ 0 & -1 \end{psmallmatrix}$, where for $c \in \C$ we write $[c]$ for the real matrix $\begin{psmallmatrix} \operatorname{Re} c & -\operatorname{Im} c \\ \operatorname{Im} c & \operatorname{Re} c \end{psmallmatrix}$ of multiplication by $c$ in the basis $(1, i)$; then $\det S = |\varphi|^2 - |\psi|^2$. Throughout, $[\,\cdot\,]$ retains this meaning, and $e := (1, i)$ denotes the \emph{ambient covector}, so that $e\, [c] = c\, e$ and $e \cdot (u,v) = u + iv$.

The action (D) was treated in \cite{mass}, and (R) is absorbed identically by the skeletonization (this is immediate, and re-proved below in a sharper form: every construction of this paper is built from a presentation-independent normal form, Section~\ref{sec:eigenpairing}). The content of Theorem~\ref{thm:main} is the action (U): substituting \eqref{eq:gen-sub} into a BV equation destroys the BV form --- the conjugate-derivative term $\bar v'_{\bar z}$ generated by $\psi \neq 0$ has no slot in \eqref{eq:BV}, which is why the complex gauges are precisely the stabilizer of the form --- but the underlying real system remains elliptic, the pipeline applies to it, and the theorem asserts that the output returns to the gauge orbit of the input, with the explicit gauge \eqref{eq:universal-gauge}.

\section{The eigenpairing equation}\label{sec:eigenpairing}

Every first-order $2\times 2$ system with invertible $\partial_x$-coefficient matrix has a unique presentation
\begin{equation}\label{eq:normalized}
V_x + N\, V_y + P\, V = h, \qquad V = (u, v),
\end{equation}
obtained by left multiplication with the inverse of that matrix; in skeleton form,
\begin{equation}\label{eq:N-def}
N = \sigma(e_1)^{-1}\sigma(e_2), \qquad
P = \sigma(e_1)^{-1}\begin{psmallmatrix} a_{13} & a_{14} \\ a_{23} & a_{24} \end{psmallmatrix}, \qquad
h = \sigma(e_1)^{-1}(f_1, f_2),
\end{equation}
where $\sigma$ is the principal symbol, $\sigma(e_1) = \begin{psmallmatrix} a_{11} & 0 \\ a_{21} & 1 \end{psmallmatrix}$, $\sigma(e_2) = \begin{psmallmatrix} a_{12} & -1 \\ a_{22} & 0 \end{psmallmatrix}$, and $e_1, e_2$ is the coordinate frame of covectors. Since any row recombination cancels in the normalization --- for $M \cdot L$ the $\partial_x$-coefficient matrix is $M$ times that of $L$, and $(M\,\cdot\,)^{-1}(M \cdot L)$ collapses --- the presentation \eqref{eq:normalized} depends only on the (R)-class of the system, and every construction of this paper, being built from $(N, P, h)$, absorbs (R) identically. The regularity inherited from the admissible class is $N \in C^1$, $P, h \in C^0$.

The endomorphism field $N$ of the unknown bundle is the spectral protagonist:

\begin{lemma}[Spectral data of the symbol]\label{lem:symbol}
For an admissible system in skeleton form,
\begin{equation}\label{eq:conic-identity}
\det \sigma(-s, 1) \;=\; a_{11}\,\bigl(s^2 + \beta s + \alpha\bigr) \;=\; a_{11} \det(N - sI),
\end{equation}
so the characteristic polynomial of $N$ is the structure polynomial $p(s)$ of the pipeline; ellipticity ($4\alpha - \beta^2 > 0$) is the condition that $N$ has no real eigenvalues, and its eigenvalues $\lambda \in \HH^+$, $\bar\lambda$ are simple. Under (R), $\sigma \mapsto M\sigma$ leaves $N$ untouched; under (U), $\sigma \mapsto \sigma S$ conjugates it, $N \mapsto S^{-1} N S$, while $\det\sigma(\xi)$ is multiplied by $\det S$. Hence $(\alpha, \beta)$, and with them $\lambda$ and $\mu$, are invariant under all of (R) and (U).
\end{lemma}

\begin{proof}
In skeleton form $\det\sigma(\xi) = a_{11}\xi_1^2 + (a_{12}+a_{21})\xi_1\xi_2 + a_{22}\xi_2^2$; evaluating at $(-s, 1)$ and dividing by $a_{11}$ gives $s^2 + \beta s + \alpha$ with the definitions of Section~\ref{sec:setup}, and $\det\sigma(-s,1) = \det\sigma(e_1)\det(N - sI)$ with $\det\sigma(e_1) = a_{11}$. Real eigenvalues of $N$ are real roots of $p$, excluded by $4\alpha - \beta^2 > 0$. The transformation laws are immediate from \eqref{eq:N-def} and multiplicativity of the determinant; $(\alpha, \beta)$ are ratios of conic coefficients, insensitive to the scalings by $\det M$ and $\det S$, and $\lambda$, $\mu$ are functions of $(\alpha, \beta)$ alone.
\end{proof}

A \emph{$\lambda$-eigencovector field} of \eqref{eq:normalized} is a nowhere-vanishing $\omega \in C^1(\Omega; \C^2)$ with $\omega N = \lambda\, \omega$ (a row vector). The central construction of the paper attaches a BV equation to every such field.

\begin{lemma}[The eigenpairing equation]\label{lem:eigen-eq}
Let \eqref{eq:normalized} be elliptic and let $\omega$ be a $\lambda$-eigencovector field. Then:
\begin{itemize}
\item[(i)] The matrix $\Xi_\omega$ with rows $\omega, \bar\omega$ is invertible at every point.
\item[(ii)] Define the \emph{connection covector} and the BV data
\begin{equation}\label{eq:E-data}
\varrho(\omega) := \omega_x + \lambda\, \omega_y - \omega P, \qquad
(\Acal, \Bcal) := -\frac{\varrho(\omega)\, \Xi_\omega^{-1}}{1 - i\lambda}, \qquad
\Fcal := \frac{\omega \cdot h}{1 - i\lambda},
\end{equation}
all in $C^0(\Omega; \C)$, with $\mu = (\lambda - i)/(\lambda + i)$. Then for every $V \in C^1(\Omega; \R^2)$, with $w := \omega \cdot V$,
\begin{equation}\label{eq:E-residual}
w_{\bar z} - \mu\, w_z + \Acal\, w + \Bcal\, \bar w - \Fcal \;=\; \frac{1}{1 - i\lambda}\; \omega \cdot \bigl( V_x + N V_y + P V - h \bigr)
\end{equation}
identically in the first jet of $V$. We write $E(\omega) := (\mu, \Acal, \Bcal, \Fcal) \in \BV(\Omega)$ for this equation, the \emph{eigenpairing equation} of $\omega$.
\item[(iii)] Since the rows of $\Xi_\omega$ are pointwise independent by (i), identity \eqref{eq:E-residual} exhibits the residual of $E(\omega)$ as a pointwise-invertible complex combination of the two real residuals: $V$ solves the system if and only if $w = \omega V$ solves $E(\omega)$.
\end{itemize}
\end{lemma}

\begin{proof}
(i) If $\Xi_\omega$ were singular at a point, $\omega$ would there be a complex multiple of a real covector $\omega_r \neq 0$; then $\omega_r N = \lambda\, \omega_r$ with $N$ real forces $\lambda \in \R$, contradicting ellipticity.

(ii) The operator identity $\partial_x + \lambda\, \partial_y = (1 - i\lambda)\bigl(\bar\partial - \mu\, \partial\bigr)$ follows from $\partial_x = \partial + \bar\partial$, $\partial_y = i(\partial - \bar\partial)$ and
\[
\frac{1 + i\lambda}{1 - i\lambda} \;=\; \frac{i(\lambda - i)}{-i(\lambda + i)} \;=\; -\mu .
\]
By the Leibniz rule and $\lambda\, \omega V_y = \omega N V_y$,
\[
w_x + \lambda\, w_y \;=\; (\omega_x + \lambda\, \omega_y)\, V + \omega\,(V_x + N V_y) \;=\; \varrho(\omega)\, V + \omega\,\bigl(V_x + N V_y + P V\bigr).
\]
Since $V$ is real, $(w, \bar w)^{\mathsf T} = \Xi_\omega V$, so $V = \Xi_\omega^{-1}(w, \bar w)^{\mathsf T}$ and the definition of $(\Acal, \Bcal)$ gives $\Acal w + \Bcal \bar w = -(1 - i\lambda)^{-1}\varrho(\omega)\, V$. Dividing the display by $(1 - i\lambda)$ and subtracting $\Fcal$ yields \eqref{eq:E-residual}. The coefficients lie in $C^0$ because $\omega, \lambda \in C^1$ and $P, h \in C^0$. (iii) is immediate from (i) and \eqref{eq:E-residual}.
\end{proof}

The companion lemma is the pointwise linear-algebra fact that will replace, everywhere below, any appeal to the existence or abundance of solutions.

\begin{lemma}[Determinacy]\label{lem:determinacy}
Two equations of BV form with the same Beltrami coefficient whose residuals agree on the pairing $w = \omega V$, identically in the first jet of $V$, for a single eigencovector field $\omega$, are equal.
\end{lemma}

\begin{proof}
The difference of the two residuals is $a w + b \bar w + c$, with no derivative terms since the principal parts coincide. Composed with the pairing it reads $(a\,\omega + b\,\bar\omega)\, V + c$, which vanishes for all values $V(z_0) \in \R^2$ at a point only if $a\,\omega + b\,\bar\omega = 0$ and $c = 0$ there; the rows of $\Xi_\omega$ are independent by Lemma~\ref{lem:eigen-eq}(i), so $a = b = c = 0$.
\end{proof}

\section{Gauge covariance and substitution naturality}\label{sec:naturality}

The eigencovector line of the simple eigenvalue $\lambda$ admits two natural operations: rescaling a section by a nowhere-vanishing complex field, and transporting it under an unknown recombination. The next two lemmas compute the effect of each on the eigenpairing equation; together they say that $\omega \mapsto E(\omega)$ converts the eigencovector ambiguity into precisely the gauge action, and nothing else.

\begin{lemma}[Gauge covariance]\label{lem:gauge-cov}
For $\phi \in C^1(\Omega; \C^*)$,
\[
E(\phi\, \omega) \;=\; \phi \cdot E(\omega)
\]
in the sense of the gauge action \eqref{eq:gauge-action}, with the pairings related by $(\phi\omega) V = \phi\,(\omega V)$.
\end{lemma}

\begin{proof}
One computes $\varrho(\phi\omega) = \phi\, \varrho(\omega) + (\phi_x + \lambda \phi_y)\, \omega$ and $\Xi_{\phi\omega} = \operatorname{diag}(\phi, \bar\phi)\, \Xi_\omega$. Since $\omega\, \Xi_\omega^{-1} = (1, 0)$,
\[
\varrho(\phi\omega)\, \Xi_{\phi\omega}^{-1}
= \Bigl( \bigl(\varrho(\omega)\Xi_\omega^{-1}\bigr)_1 + \frac{\phi_x + \lambda\phi_y}{\phi},\;\; \frac{\phi}{\bar\phi}\, \bigl(\varrho(\omega)\Xi_\omega^{-1}\bigr)_2 \Bigr),
\]
and $\phi_x + \lambda\phi_y = (1 - i\lambda)(\phi_{\bar z} - \mu\, \phi_z)$ by the operator identity of Lemma~\ref{lem:eigen-eq}. Multiplying by $-(1 - i\lambda)^{-1}$ yields $\Acal' = \Acal - \phi_{\bar z}/\phi + \mu\, \phi_z/\phi$ and $\Bcal' = \Bcal\, \phi/\bar\phi$; and $\Fcal' = \phi\, \Fcal$ directly from \eqref{eq:E-data}. These are the laws \eqref{eq:gauge-action}.
\end{proof}

\begin{lemma}[Substitution naturality]\label{lem:naturality}
Let $S \in C^1(\Omega; \mathrm{GL}(2,\R))$, of either orientation, and transform the system \eqref{eq:normalized} by $V = S V'$. Its normalized data become
\[
N' = S^{-1} N S, \qquad P' = S^{-1}\bigl(S_x + N S_y + P S\bigr), \qquad h' = S^{-1} h,
\]
the field $\omega' := \omega S$ is a $\lambda$-eigencovector of $N'$, and
\[
E'(\omega S) \;=\; E(\omega)
\]
--- identical data $(\mu, \Acal, \Bcal, \Fcal)$ --- with the pairings related by $\omega' V' = \omega V$.
\end{lemma}

\begin{proof}
Substituting $V = SV'$ into \eqref{eq:normalized} gives $S V'_x + N S V'_y + (S_x + N S_y + P S)\,V' = h$; left multiplication by $S^{-1}$ is the normalization, producing the displayed data. The eigencovector property survives, with the same eigenvalue: $\omega S\, (S^{-1} N S) = \omega N S = \lambda\, \omega S$; in particular $\lambda$ and $\mu$ are unchanged, as Lemma~\ref{lem:symbol} already showed at the level of the symbol. The connection covector transforms tensorially:
\[
\varrho'(\omega S) \;=\; (\omega S)_x + \lambda\, (\omega S)_y - \omega S\, S^{-1}\bigl(S_x + N S_y + P S\bigr).
\]
Expanding $(\omega S)_x = \omega_x S + \omega S_x$ and $(\omega S)_y = \omega_y S + \omega S_y$, the two occurrences of $\omega S_x$ cancel, the terms $\lambda\, \omega S_y - \omega N S_y$ vanish by the eigencovector relation, and what remains is
\[
\varrho'(\omega S) \;=\; \bigl(\omega_x + \lambda\, \omega_y - \omega P\bigr)\, S \;=\; \varrho(\omega)\, S .
\]
Finally $\Xi_{\omega S} = \Xi_\omega S$ (rows $\omega S$ and $\overline{\omega S} = \bar\omega S$, as $S$ is real), so
\[
\varrho'(\omega S)\, \Xi_{\omega S}^{-1} = \varrho(\omega)\, S\, S^{-1}\, \Xi_\omega^{-1} = \varrho(\omega)\, \Xi_\omega^{-1},
\qquad \omega' h' = \omega S S^{-1} h = \omega h,
\]
and the data \eqref{eq:E-data} coincide. The pairing relation $\omega S V' = \omega V$ is the definition of $V'$.
\end{proof}

\section{Two identifications}\label{sec:identifications}

The machinery of Sections~\ref{sec:eigenpairing}--\ref{sec:naturality} connects to the pipeline through two identifications: the pipeline output is the eigenpairing equation in the pipeline's normalization of the eigencovector, and the original BV equation is the eigenpairing equation of its own realification in the ambient normalization.

\begin{proposition}[The bundling is a spectral pairing]\label{prop:bundling}
In the skeleton frame, the covector
\[
\omega_0 \;:=\; \bigl(a_{22} - \lambda\, a_{12},\; \lambda\bigr)
\]
is a $\lambda$-eigencovector field of $N$, of class $C^1$, and its pairing with the unknown vector is the pipeline's complex unknown:
\[
\omega_0 \cdot (u, v) \;=\; a_{22}\,u + \lambda\,(v - a_{12}\,u) \;=\; U + \lambda V \;=\; w .
\]
\end{proposition}

\begin{proof}
From \eqref{eq:N-def},
\[
N = \frac{1}{a_{11}}\begin{pmatrix} a_{12} & -1 \\ a_{11}a_{22} - a_{12}a_{21} & a_{21} \end{pmatrix}.
\]
The second component of $\omega_0 N$ is $\bigl[-(a_{22} - \lambda a_{12}) + \lambda a_{21}\bigr]/a_{11} = -(\alpha + \beta\lambda) = \lambda^2$, using the structure polynomial $\lambda^2 + \beta\lambda + \alpha = 0$; the first component of $\omega_0 N - \lambda\,\omega_0$ equals $a_{12}\,\bigl(\alpha + \beta\lambda + \lambda^2\bigr) = 0$ by the same identity. Regularity is that of $a_{12}, a_{22}, \lambda \in C^1$, and $\omega_0$ is nowhere zero since its second component is $\lambda \in \HH^+$. The displayed pairing is the pipeline's canonical substitution $U = a_{22}u$, $V = v - a_{12}u$ followed by its spectral bundling $w = U + \lambda V$.
\end{proof}

\begin{proposition}[Pipeline identification]\label{prop:pipeline-id}
For every admissible system $L$ in skeleton form,
\[
\pi(L) \;=\; E_L(\omega_0).
\]
\end{proposition}

\begin{proof}
By Proposition~\ref{prop:bundling} the pipeline's unknown is $w = \omega_0 V$, and the residual identity \eqref{eq:residual-identity} states, identically in the first jet of $V$,
\[
\mathrm{Res}_{\pi(L)}(\omega_0 V) \;=\; \tfrac12(1-\mu)\,(-\lambda^2,\; \lambda)\cdot R,
\]
where $R = (R_1, R_2)$ are the skeleton residuals. Two pointwise algebraic identities convert the right-hand side into the eigenpairing residual \eqref{eq:E-residual}. The prefactors agree: $\mu = (\lambda - i)/(\lambda + i)$ gives
\[
\tfrac12(1 - \mu) \;=\; \frac{i}{\lambda + i} \;=\; \frac{1}{1 - i\lambda} .
\]
The covectors agree: the normalized residual of $L$ is $\sigma(e_1)^{-1} R$, and
\[
\omega_0\, \sigma(e_1)^{-1} \;=\; \Bigl( \frac{a_{22} - \lambda(a_{12} + a_{21})}{a_{11}},\; \lambda \Bigr) \;=\; \bigl(\alpha + \beta\lambda,\; \lambda\bigr) \;=\; \bigl(-\lambda^2,\; \lambda\bigr)
\]
by the structure polynomial. Hence
\[
\mathrm{Res}_{\pi(L)}(\omega_0 V) \;=\; \frac{1}{1 - i\lambda}\;\omega_0 \cdot \bigl(V_x + N V_y + P V - h\bigr) \;=\; \mathrm{Res}_{E(\omega_0)}(\omega_0 V)
\]
identically in the jet of $V$. Both equations are of BV form with the same Beltrami coefficient --- the pipeline's $\mu$ is the Cayley image of the same $\lambda$ --- so Lemma~\ref{lem:determinacy} gives $\pi(L) = E_L(\omega_0)$.
\end{proof}

\begin{lemma}[Realification]\label{lem:realification}
Let $D = (\mu, \Acal, \Bcal, \Fcal) \in \BV(\Omega)$ and let $L_0$ be its realification: the pair of real equations $(R_1, R_2) = (\operatorname{Re}, \operatorname{Im})\,\mathrm{Res}_D$ in the unknowns $(u, v)$, $w = u + iv$. Then $L_0$ is elliptic, its normalized form has
\[
N_0 \;=\; [\lambda], \qquad \lambda \;=\; i\,\frac{1 + \mu}{1 - \mu} \;\in\; \HH^+
\]
--- the canonical endomorphism of a realified BV equation is multiplication by its own spectral root --- the ambient covector $e = (1, i)$ is a $\lambda$-eigencovector of $N_0$, and
\[
E_{L_0}(e) \;=\; D .
\]
\end{lemma}

\begin{proof}
The principal part of $\mathrm{Res}_D$ is $w_{\bar z} - \mu w_z = c_x\, w_x + c_y\, w_y$ with $c_x = \tfrac12(1 - \mu)$, $c_y = \tfrac{i}{2}(1 + \mu)$, and $w_x = e\, V_x$, $w_y = e\, V_y$. Splitting into real and imaginary parts, the coefficient matrices of $V_x$ and $V_y$ in $(R_1, R_2)$ are $[c_x]$ and $[c_y]$; $c_x \neq 0$ since $|\mu| < 1$, and the normalization gives
\[
N_0 \;=\; [c_x]^{-1}[c_y] \;=\; [c_y/c_x] \;=\; [\lambda],
\]
since $c_y/c_x = i(1+\mu)/(1-\mu) = \lambda$. As $\lambda \notin \R$, the system is elliptic, and $e\, N_0 = e\,[\lambda] = \lambda\, e$: the ambient bundling $w = u + iv$ is the spectral pairing of $N_0$ in the normalization $e$. For the identification, the normalized residual of $L_0$ is $[c_x]^{-1}(R_1, R_2)$, and
\[
\mathrm{Res}_{E(e)}(e\, V) \;=\; \frac{1}{1 - i\lambda}\; e\,[c_x]^{-1}\,(R_1, R_2) \;=\; \frac{c_x^{-1}}{1 - i\lambda}\,\bigl(R_1 + i R_2\bigr) \;=\; \frac{c_x^{-1}}{1 - i\lambda}\;\mathrm{Res}_D(w),
\]
using $e\,[c_x]^{-1} = c_x^{-1} e$ and $e \cdot (R_1, R_2) = R_1 + i R_2$. The prefactor is exactly $1$: $(1 - i\lambda)^{-1} = \tfrac12(1 - \mu) = c_x$, as in Proposition~\ref{prop:pipeline-id}. So the residuals of $E_{L_0}(e)$ and of $D$ agree identically on the pairing, both equations are of BV form with the same $\mu$, and Lemma~\ref{lem:determinacy} concludes.
\end{proof}

\section{Proof of the Absorption Theorem}\label{sec:proof}

\begin{proof}[Proof of Theorem~\ref{thm:main}]
Realification and substitution commute: substituting $w = \varphi v' + \psi \bar v'$ into the equation of $D$ and separating real and imaginary parts is the same as separating first and then substituting $V = S V'$, where $S = [\varphi] + [\psi]\begin{psmallmatrix} 1 & 0 \\ 0 & -1 \end{psmallmatrix}$ is the real matrix of $z \mapsto \varphi z + \psi \bar z$, with $\det S = |\varphi|^2 - |\psi|^2 > 0$ and $S \in C^1$. So, in the notation of Lemmas~\ref{lem:naturality} and~\ref{lem:realification}, $S^\sharp D$ is the system $L_0$ transformed by $V = SV'$. It is admissible: its determinant conic is that of $L_0$ scaled by the positive factor $\det S$ (Lemma~\ref{lem:symbol}), so the pipeline applies to its skeletonization, and skeletonization is an (R)-move, under which the normalized form \eqref{eq:normalized} --- hence every eigenpairing equation --- is unchanged.

Now assemble the identifications. First, by Proposition~\ref{prop:pipeline-id} applied to the skeletonization of $S^\sharp D$,
\[
D_S \;=\; \pi\bigl(S^\sharp D\bigr) \;=\; E_{S^\sharp D}\bigl(\omega_0'\bigr),
\]
where $\omega_0'$ is the pipeline normalization of Proposition~\ref{prop:bundling} for the substituted skeleton, with second component $(\omega_0')_2 = \lambda$ --- the same function $\lambda$, by Lemma~\ref{lem:symbol}. Second, by Lemma~\ref{lem:realification} and Lemma~\ref{lem:naturality},
\[
D \;=\; E_{L_0}(e) \;=\; E_{S^\sharp D}\bigl(e\, S\bigr), \qquad e\, S \;=\; \bigl(\varphi + \psi,\; i(\varphi - \psi)\bigr),
\]
the ambient covector computed from $e\,[\varphi] = \varphi\, e$ and $e\,[\psi]\begin{psmallmatrix} 1 & 0 \\ 0 & -1 \end{psmallmatrix} = \psi\,(1, -i)$. Third, $e S$ and $\omega_0'$ are both $\lambda$-eigencovectors of the canonical endomorphism $S^{-1} N_0 S$ of the substituted system, and $\lambda$ is a simple eigenvalue by ellipticity (Lemma~\ref{lem:symbol}), so they are pointwise proportional,
\[
e\, S \;=\; \tilde\phi\; \omega_0', \qquad \tilde\phi \;=\; \frac{(e S)_2}{(\omega_0')_2} \;=\; \frac{i(\varphi - \psi)}{\lambda},
\]
read off the second components; $\tilde\phi$ is $C^1$ and nowhere zero, since $\varphi = \psi$ at a point would force $\det S = 0$ there. Finally, by Lemma~\ref{lem:gauge-cov},
\[
D \;=\; E_{S^\sharp D}(e S) \;=\; E_{S^\sharp D}\bigl(\tilde\phi\, \omega_0'\bigr) \;=\; \tilde\phi \cdot E_{S^\sharp D}(\omega_0') \;=\; \tilde\phi \cdot D_S,
\]
hence $D_S = \tilde\phi^{-1} \cdot D$ with
\[
\tilde\phi^{-1} \;=\; \frac{\lambda}{i(\varphi - \psi)} \;=\; \frac{-i\lambda}{\varphi - \psi} \;=\; \tilde\varphi .
\]
The pointwise consequences $\mu_S = \mu$, $|\Bcal_S| = |\Bcal|$, $\Theta_{D_S} = \Theta_D$ are the gauge laws \eqref{eq:gauge-action} for $\tilde\varphi$.
\end{proof}

\begin{remark}[Universality of the gauge]\label{rem:universal-gauge}
The gauge depends on the substitution and, through $\lambda$, on the Beltrami coefficient of $D$, and on nothing else --- not on $\Acal$, $\Bcal$, $\Fcal$, nor on any derivative of the data. It is the ratio of the two eigencovector normalizations in play: the ambient one transported by the substitution and the spectral one selected by the pipeline. Three special cases deserve note. For the identity substitution $(\varphi, \psi) = (1, 0)$ the gauge is $-i\lambda$, generically non-constant: even re-running the pipeline on the untouched realification of $D$ returns not $D$ itself but its gauge image $(-i\lambda) \cdot D$ --- the pipeline's spectral normalization $(\omega_0)_2 = \lambda$ differs from the ambient normalization $e_2 = i$, and the normal form remembers this. And for $\C$-linear substitutions $\psi \equiv 0$ the gauge is $-i\lambda/\varphi$: the classical gauge $\varphi^{-1}$ composed with the self-consistency gauge $-i\lambda$.
\end{remark}

\begin{corollary}[Invariants of the real system]\label{cor:invariants}
Every gauge--diffeomorphism invariant of BV data is an invariant of the underlying real elliptic system under the full symmetry groupoid $\Scal$ of Definition~\ref{def:groupoid}. In particular the density $\Theta$, the pseudo-analytic mass $\Mcal$, the zero locus and dichotomy of $\Bcal$, the analytic class $\Bcal \equiv 0$, are invariants of the system: (R) acts trivially on the normal form, (U) acts through the gauge \eqref{eq:universal-gauge}, and (D) acts by the covariance of \cite[\S 5]{mass}.
\end{corollary}

\begin{proof}
(R) is absorbed by the normalization \eqref{eq:normalized} as observed in Section~\ref{sec:eigenpairing}; (U) is Theorem~\ref{thm:main}; (D) is \cite{mass}. The listed quantities are gauge--diffeomorphism invariant by \cite{mass}.
\end{proof}

\begin{remark}[Mass invariance, explicitly]\label{rem:explicit-mass}
For the pseudo-analytic mass the invariance under (U) admits a fully explicit form, and a verification independent of this paper's machinery. Theorem~\ref{thm:main} and the gauge laws give the pointwise density identity
\[
\Theta_{D_S} \;=\; \frac{|\Bcal_S|^2}{1 - |\mu_S|^2}\,dx\,dy \;=\; \Theta_D,
\qquad
\Bcal_S = \frac{\tilde\varphi}{\overline{\tilde\varphi}}\,\Bcal
= -\frac{\lambda\,(\bar\varphi - \bar\psi)}{\bar\lambda\,(\varphi - \psi)}\,\Bcal,
\quad \mu_S = \mu,
\]
with the $\Bcal$-factor unimodular, so $\Mcal(D_S) = \Mcal(D)$ follows by integration with no change of variables. The identity has moreover been certified through the \emph{literal pipeline formulas} of \cite[\S 3]{mass} alone: running the printed seven steps on the substituted system, one finds $\alpha_S = |\lambda|^2$ and $\beta_S = -2\operatorname{Re}\lambda$ exactly --- the structure data, and with them the $\Theta$-denominator, are equal before any further computation --- while the displayed transformation of $\Bcal_S$, and $\Fcal_S = \tilde\varphi\, \Fcal$, hold as exact rational identities in the data and their first jets. The pseudo-analytic mass of a real elliptic system can thus be confirmed frame-independent by direct computation, without uniformization and without the eigenpairing construction.
\end{remark}

\begin{corollary}[Orientation-reversing substitutions]\label{cor:reversal}
Let $|\psi| > |\varphi|$ pointwise, so $\det S < 0$. The substituted system is elliptic with its determinant conic negated; its skeleton form has $a_{11} < 0$, outside the admissible class, but the structure data still satisfy $4\alpha - \beta^2 > 0$ and the pipeline formulas extend verbatim with the upper-half-plane root $\lambda \in \HH^+$. With $\pi$ so extended --- equivalently, with $D_S := E_{S^\sharp D}(\omega_0')$ --- the conclusion of Theorem~\ref{thm:main} holds unchanged:
\[
D_S \;=\; \frac{-i\lambda}{\varphi - \psi} \cdot D .
\]
In particular $\mu_S = \mu$ (not $\bar\mu$), $|\Bcal_S| = |\Bcal|$, and $\Theta$ is preserved exactly: the symmetry groupoid of Corollary~\ref{cor:invariants} extends to $\mathrm{GL}(2,\R)$-valued substitutions with nowhere-vanishing determinant.
\end{corollary}

\begin{proof}
The conic is scaled by $\det S < 0$ (Lemma~\ref{lem:symbol}), which negates $a_{11}$ but leaves the ratios $(\alpha, \beta)$, hence $\lambda$ and the discriminant condition, untouched; the skeleton form, structure data, and the eigencovector $\omega_0'$ of Proposition~\ref{prop:bundling} are therefore still defined, and the extended $\pi$ equals $E(\omega_0')$ by the proof of Proposition~\ref{prop:pipeline-id}, which nowhere used the sign of $a_{11}$. Lemmas~\ref{lem:eigen-eq}--\ref{lem:realification} and the proof of Theorem~\ref{thm:main} are insensitive to the sign of $\det S$ and proportionality of eigencovectors with $(\omega_0')_2 = \lambda$, $(eS)_2 = i(\varphi - \psi) \neq 0$ is unchanged, so the identical assembly goes through.
\end{proof}

\begin{remark}\label{rem:reversal-meaning}
The corollary is sharper than one might first expect. At the level of the form, a substitution with $\psi$ dominant composes the equation with the conjugation $w \mapsto \bar w$, and one might predict $\mu_S = \bar\mu$ --- the conjugate conformal structure. Re-normalization re-selects the $\HH^+$ root of the structure polynomial and undoes the conjugation: the entire invariant theory, not only $|\mu|$, survives orientation reversal of the unknown frame.
\end{remark}

\section{Concluding remarks}\label{sec:concluding}

The proof re-founds the pipeline of \cite{mass} as the coordinate expression of a frame-independent spectral construction. The principal symbol defines the endomorphism $N$ before any normalization is performed; its $\lambda$-eigencovector line is the invariant object; the eigenpairing equation $E(\omega)$ is the value of the construction at a section $\omega$ of that line; and the two normalizations that occur in practice --- the pipeline's spectral section $\omega_0$, with $(\omega_0)_2 = \lambda$, and the ambient section $e = (1, i)$ carried by the complex notation itself --- differ by the universal gauge \eqref{eq:universal-gauge}. The gauge group of \cite{mass} thus appears a priori, as the normalization ambiguity of an eigencovector, and the absorption theorem is the statement that the construction never chose a real frame to begin with.

Three threads continue from here. First, the connection covector $\varrho(\omega) = \omega_x + \lambda\omega_y - \omega P$ of \eqref{eq:E-data} is the intrinsic form of the pipeline's derivative terms, and Lemma~\ref{lem:naturality} says it transforms tensorially along the eigencovector line; pairing the unknowns with a covector \emph{off} the eigencovector line instead produces a \emph{framed} Beltrami--Vekua class, where the corresponding connection term is the $L$-Wronskian of the frame and the pseudo-analytic mass admits a closed substitution-invariant extension. The present theorem identifies the pipeline with the algebraic frame-straightening of that class, settling the reformulation proposed there. Second, the spectral picture suggests degree-theoretic readings of the remaining invariants --- the divisor of $\Bcal$ as data of the eigenpairing, and the holonomy of $\Acal$ on multiply connected domains as the monodromy of the eigencovector normalization; with Corollary~\ref{cor:invariants} these are now well-defined questions about real elliptic systems. Third, Corollary~\ref{cor:reversal} extends the symmetry groupoid to $\mathrm{GL}(2,\R)$, and it is natural to ask for the corresponding extension on the diffeomorphism side --- orientation-reversing coordinate changes --- and for the joint classification.

\subsection*{Use of Generative AI Tools}
\medskip

The author discloses the use of Anthropic's Claude (Claude Fable 5, accessed
through the Claude.ai web interface in June 2026) in the preparation of this
manuscript. The tool was used as follows:

\begin{enumerate}
\item[(i)] \emph{Exploratory dialogue.} The proof strategy --- the eigenpairing
equation, its naturality, and the identification of the pipeline with the
spectral pairing --- emerged in a collaborative research session, building on
the conjecture, partial results, and numerical protocols which were developed in earlier sessions. The universality of the gauge
\eqref{eq:universal-gauge} and the orientation-reversing extension were
discovered in the same dialogue.

\item[(ii)] \emph{Symbolic verification.} All lemmas of this paper, and the
closed form of the gauge, were verified by exact rational computer algebra on
generic variable-coefficient instances and generic variable substitutions of
both orientation classes, with scripts produced with the tool's assistance and
re-run by the author; the verification script accompanies this paper as an
ancillary file.

\item[(iii)] \emph{Drafting and revision of prose.} The manuscript was drafted
in iterative dialogue; all claims and their precise wording were reviewed by
the author.
\end{enumerate}

All mathematical statements are presented with complete proofs that the author
has verified independently. The author takes full responsibility for the
accuracy, originality, and integrity of all content.

\subsection*{Disclosure of interest}

The author reports there are no competing interests to declare.

\bibliographystyle{plain}

\end{document}